\documentclass[11pt]{article}
    \usepackage{amssymb, latexsym, amsmath, xcolor, amsthm}
    \usepackage{hyperref}
    \usepackage[capitalize]{cleveref}
        \newcommand{\setleftmargin}[1]{
       \addtolength{\textwidth}{\oddsidemargin}
       \addtolength{\textwidth}{1in}
       \addtolength{\textwidth}{-#1}
       \setlength{\oddsidemargin}{-1in}
       \addtolength{\oddsidemargin}{#1}
       \setlength{\evensidemargin}{\oddsidemargin}
    }
    \newcommand{\setrightmargin}[1]{
       \setlength{\textwidth}{8.5in}
       \addtolength{\textwidth}{-\oddsidemargin}
       \addtolength{\textwidth}{-1in}
       \addtolength{\textwidth}{-#1}
    }
    \newcommand{\settopmargin}[1]{
       \addtolength{\textheight}{\topmargin}
       \addtolength{\textheight}{1in}
       \addtolength{\textheight}{\headheight}
       \addtolength{\textheight}{\headsep}
       \addtolength{\textheight}{-#1}
       \setlength{\topmargin}{-1in}
       \addtolength{\topmargin}{-\headheight}
       \addtolength{\topmargin}{-\headsep}
       \addtolength{\topmargin}{#1}
    }
    \newcommand{\setbottommargin}[1]{
       \setlength{\textheight}{11in}
       \addtolength{\textheight}{-\topmargin}
       \addtolength{\textheight}{-1in}
       \addtolength{\textheight}{-\footskip}
       \addtolength{\textheight}{-#1}
    }
    \newcommand{\setallmargins}[1]{
       \settopmargin{#1}
       \setbottommargin{#1}
       \setleftmargin{#1}
       \setrightmargin{#1}
    } \setallmargins{1.2in}
  \newcommand{\tung}[1]{{\color{blue}#1}}
  \newcommand{\ord}{\text{ord}}
  \newcommand{\F}{\mathbb{F}}

    \newcommand{\Z}{\mathbb{Z}}

\title{A Note on Finite Number Rings 
\thanks{
{\it Mathematics Subject Classification}:  11T06, 11T30, 11Z05, 97H40 }}
\author{Suk-Geun Hwang \thanks{
Department of Mathematics Education, Kyungpook
University, Taegu 702-701, Korea and Department of Mathematics, University of
Wisconsin-Whitewater, Whitewater, WI 53190. USA, Email:
sghwang@knu.ac.kr, hwangs@uww.edu}, Woo Jeon \thanks{ Department of Mathematics \& Statistics, University of Wisconsin-Green Bay, Green Bay, WI 54311, USA, Email: jeonw@uwgb.edu}, and
Ki-Bong Nam \thanks{ Dept. of Math., Univ. of Wisconsin-Whitewater,
Whitewater, WI 53190, USA, Email: namk@uww.edu}, Tung T. Nguyen, \thanks {Dept. of Kinesiology, Western University,
Ontario, Canada, Email: tungnt@uchicago.edu} }

\begin{document}
\maketitle
  \begin{abstract}
We define the finite number ring ${\Bbb Z}_n [\sqrt [m] r]$ where $m,n$ are positive integers and $r$ in an integer akin to the definition of the Gaussian integer ${\Bbb Z}[i]$. This idea is also introduced briefly in \cite{Weil}. By definition, this finite number ring ${\Bbb Z}_n [\sqrt [m] r]$ is naturally isomorphic to the ring
${\Bbb Z}_n[x]/{\langle x^m-r \rangle}$. From an educational standpoint, this description offers a straightforward and elementary presentation of this finite ring, making it suitable for readers who do not have extensive exposure to abstract algebra. We discuss various arithmetical properties of this ring. In particular, when $n=p$ is a prime number
and $\Z_p$ contains a primitive $m$-root of unity, we describe the structure of $\Z_n[\sqrt[m]{r}]$ explicitly. 
  \end{abstract}

    \newtheorem{lemma}{Lemma}
    \newtheorem{prop}{Proposition}
    \newtheorem{thm}{Theorem}
    \newtheorem{coro}{Corollary}
    \newtheorem{definition}{Definition}
 \newtheorem{exa}{Example}[section]
  \newtheorem{rem}{Remark}[section]

 \newtheorem{cor}[thm]{Corollary}

\newtheorem{note}{Note}[section]

Keywords: Radical Extension, Unital Determinant, Finite Fields, Polynomial Factorization
\section{Introduction}

For each positive integer $n$, we will denote by $\Z_n$ the ring of integers modulo $n$. We recommend \cite{J, L} for some good textbooks on the ring $\Z_n$. Throughout the paper, $p$ is a prime number.

Recall that the Gaussian integers ${\Bbb Z}[i]$ is isomorphic to ${\Bbb Z}[x]/\langle x^2+1\rangle$
where $\langle x^2+1 \rangle$ is the ideal of ${\Bbb Z}[x]$ which is generated by $x^2+1$ (see \cite{I, ST}). Note also that $\Z[i]$ can also be considered as the set of all elements of the form $a+b i$ where $a, b \in \Z$ and $i^2=-1.$ Addition and multiplication on $\Z[i]$ can be defined in a natural way. For example 

\[ (a+bi)(c+di)=(ac-bd)+(ad+bc)i. \]
The advantage of this presentation of $\Z[i]$ is that it is rather straightforward and can be understood by readers who do not have serious exposure to the theory of quotient rings. 

Motivated by this observation, we can define an extension ring ${\Bbb Z}_n[i]$ of ${\Bbb Z}_n$ as the set of all formal elements of the form $a+bi$ where $a,b \in \Z_n$. Multiplication and addition are also defined in a natural way. More specifically, if $r \in \Z $ (or $\Z_n$) and $m$ is a positive integer, we can define 
\begin{eqnarray*}
& &{\Bbb Z}_n [ \sqrt [m] r]=\{x_0+x_1\sqrt [m] r+\cdots +
x_{m-1}(\sqrt [m] {r})^{m-1}|x_0,\cdots ,x_{m-1}\in {\Bbb Z}_n\}
\end{eqnarray*}

The addition on $\Z_n[\sqrt[m]{r}]$ is given naturally as 
\[ \sum_{i=0}^{m-1} x_i (\sqrt[m]{r})^i + \sum_{i=0}^{m-1} x_i (\sqrt[m]{r})^i= \sum_{i=0}^{m-1} (x_i+y_i) (\sqrt[m]{r})^i  \] 

Similarly, the multiplication on $\Z_n[\sqrt[m]{r}]$  can defined by the following rule (and extended by linearity) \\
\begin{equation}
(\sqrt [m] {r})^{i} (\sqrt [m] {r})^{j}= \begin{cases}(\sqrt [m] {r})^{i+j} & \text{if
$i+j<m$}\\
{ r}(\sqrt [m] { r})^{k} & \text{if
$i+j>m$ and $k\equiv i+j$ mod $m$}\\
{ r} & \text{if
$i+j=m$}
\end{cases}
\end{equation}
We can see that $\Z_n[\sqrt[m]{r}]$ is a commutative ring which is naturally isomorphic to the quotient ring $\Z_n[x]/ \langle x^m -r \rangle$. The goal of this paper is to prove some fundamental properties of these finite rings using the explicit description mentioned above.

\section{Ring theoretic properties of $\Z_n[\sqrt[m]{r}]$}

We begin with a lemma about finite commutative rings. 
\begin{lemma} \label{lem:finite_ring}
    Let $R$ be a finite commutative ring. Then $R$ is an integral domain if and only if $R$ is a field. 
\end{lemma}
\begin{proof}
    If $R$ is a field then clearly $R$ is an integral domain. Conversely, suppose that $R$ is an integral domain. We claim that each element $a \in R$ such that $a \neq 0$ has an inverse. In fact, let $m_a: R \to R$ be the multiplication map 
    \[ m_a(x) = ax.\] 
    Because $R$ is an integral domain, the map $m_a$ is injective. Since $R$ is finite, $m_a$ must be surjective as well. Therefore, we can find $b \in R$ such that $ab=1.$ This shows that $a$ is invertible in $R.$ Since this is true for all $a \in R$ such that $a \neq 0$, we conclude that $R$ is a field. 
\end{proof}

Suppose that $n=p_1^{\alpha_1} \ldots p_{s}^{\alpha_s}$ be the prime factorization of $n.$ Then, by the Chinese remainder theorem 
\[ \Z_n \cong \prod_{i=1}^{s} \Z_{p_i^{\alpha_i}}. \]
Consequently 
\[ \Z_n[\sqrt[m]r] \cong \prod_{i=1}^{s} \Z_{p_i^{\alpha_i}}[\sqrt[m]r]. \]
A direct consequence of this observation is the following 
\begin{prop} \label{prop:n_is_prime}
If $n$ is not a prime integer, then the ring ${\Bbb Z}_n[\sqrt[m]{r}]$ is not an integral domain.
\end{prop}
\begin{proof} If $n$ is not prime, then ${\Bbb Z}_n$ is not an integral domain. Since the ring ${\Bbb Z}_n[\sqrt r]$ contains ${\Bbb
Z}_n$, we conclude that is also not an integral domain.
\end{proof}
\begin{prop} \label{prop:factorization}
    If $x^m-r$ is reducible over $\Z_n$; namely there exists $f, g \in \Z_n[x]$ such that $\deg(f)>1$ and $\deg(g)>1$ and
    \[ x^m-r = f(x)g(x),\]
    then $\Z_n[\sqrt[m]{r}]$ is not an integral domain. 
\end{prop}
\begin{proof}
    By our assumption 
    \[ f(\sqrt[m]{r}) g(\sqrt[m]{r}) = (\sqrt[m]{r})^m-r=0.\]
    Since $\deg(f), \deg(g)<m$, we conclude that $f(\sqrt[m]{r})$ must be a non-zero divisor in $\Z_n[\sqrt[m]{r}].$ Therefore, $\Z_n[\sqrt[m]{r}]$ is not an integral domain.
\end{proof}
We have the following condition, which provides a necessary and sufficient condition for $\Z_n[\sqrt[m]r]$ to be a field. 
\begin{thm} \label{thm:fields_condition}
    The finite number ring $\Z_n[\sqrt[m]{r}]$ is a field if and only the following two conditions are satisfied 
    \begin{enumerate}
        \item $n$ is a prime number. 
        \item $x^m-r$ is irreducible in $\Z_n[x].$
    \end{enumerate}
\end{thm}

\begin{proof}
    The necessary part of this Theorem follows from Lemma \ref{lem:finite_ring}, Proposition \ref{prop:n_is_prime}, and Proposition \ref{prop:factorization}. Conversely, suppose that $n$ is a prime number and $x^m-r$ is irreducible over $\Z_n[r]$, we claim that $\Z_n[\sqrt[m]{r}]$ is a field. Let $a =a_0+a_1 \sqrt [m] r+a_2 (\sqrt [m] {r})^2+ \cdots + a_{m-1} (\sqrt [m] {r})^{m-1}$ be a non-zero element in $\Z_n[\sqrt[m]{r}]$.  Let 
    \[ h(x) = a_0+a_x+\cdots + a_{m-1}x^{m-1} \in \Z_p[x].\]
    Since $p$ is prime, $\Z_p$ is a field, and $\Z_p[x]$ is a Euclidean domain. By the Euclidean algorithm and the fact that $x^m-r$ is irreducible of degree $m$, we can find $f_1, f_2 \in \Z[x]$ such that 
    \[ h(x)f_1(x)+(x^m-r)f_2(x)= 1.\]
    We then see that 
    \[ a f_1(\sqrt[m]r) = h(\sqrt[m]r)f_1(\sqrt[m]r)=1.\]
    This shows that $a$ is invertible as required. 
\end{proof}
\subsection{Units in $\Z_n[\sqrt[m]{r}]$}

In this section, we study units in the ring $\Z_n[\sqrt[m]{r}]$ by embedding it in a matrix ring. We start with the following lemma.

\begin{lemma} \label{lem:det}
A non-zero element $a=a_0+a_1 \sqrt [m] r+a_2 (\sqrt [m] {r})^2+ \cdots + a_{m-1} (\sqrt [m] {r})^{m-1}$ of the
ring ${\Bbb Z}_n [ \sqrt [m] r]$ is a unit, if and only if $A$ is invertible where 
$$A =  \begin{bmatrix}a_0&a_{m-1}r& a_{m-2}r& a_{m-3}r&\cdots &a_1r\\
a_1&a_0&a_{m-1}r & a_{m-2}r&\cdots &a_2r\\
a_2&a_1&a_0&a_{m-1}r &\cdots &a_3r\\
& & & \cdots &\\
a_{m-1}&a_{m-2}&a_{m-3}&a_{m-4}&\cdots &a_0\end{bmatrix}.$$
Equivalently, $a_0+a_1 \sqrt [m] r+a_2 (\sqrt [m] {r})^2+ \cdots + a_{m-1} (\sqrt [m] {r})^{m-1}$ is a unit if and only if $ \Delta = \det(A) \in \Z_n^{\times}.$
\end{lemma}

\begin{proof}
Suppose that $A$ is invertible. We want to find  a non-zero element
\[ b_0+b_1 \sqrt [m] r+b_2 (\sqrt [m] {r})^2+ \cdots + b_{m-1} (\sqrt [m] {r})^{m-1},\]
of the ring $\Z_n[\sqrt[m]{r}]$ such that
\begin{eqnarray}\label{61}
& & (a_0+a_1 \sqrt [m] r+ \cdots + a_{m-1} (\sqrt [m] {r})^{m-1} )
(b_0+b_1 \sqrt [m] r+\cdots + b_{m-1} (\sqrt [m] {r})^{m-1} )\nonumber =1.
\end{eqnarray}
This is equivalent to 
\begin{eqnarray}\label{611}
& &a_0b_0+a_{m-1}rb_1+ a_{m-2}r b_2+ a_{m-3}r b_3+\cdots +a_1rb_{m-1}=1\nonumber\\
& &a_1b_0+a_0b_1+a_{m-1}r b_2+ a_{m-2}r b_3+\cdots +a_2rb_{m-1}=0\nonumber\\
& &a_2b_0+a_1b_1+a_0b_2+a_{m-1}rb_3+\cdots +a_3rb_{m-1}=0\nonumber\\
& & \qquad \qquad \qquad \qquad \qquad \qquad \cdots \nonumber\\
& &a_{m-1}b_0+a_{m-2}b_1+a_{m-3}b_2+a_{m-4}b_3+\cdots +a_0b_{m-1}=0.
\end{eqnarray}
Since $A$ is invertible, the system (\ref{611}) has a unique solution. 
This implies that the system of linear equations (\ref{611}) has a unique non-zero solution. We conclude that $a$ is invertible. Conversely, suppose that $a$ is invertible. Then the multiplication $m_a: \Z_n[\sqrt[m]{r}] \to \Z_n[\sqrt[m]{r}]$ is invertible. In particular, for each $0 \leq i \leq m-1$, we can find 
\[ v_i = b_{0i}+b_{1i} \sqrt[m]{r}+ \ldots+ b_{(m-1)i} (\sqrt[m]{r})^{m-1}, \]
such that $av_i = (\sqrt[m]{r})^i.$ In other words, we have 

\begin{align*}
    A \begin{bmatrix}
           b_{0i} \\
           b_{1i} \\
           \vdots \\
           b_{ii} \\
           \vdots \\
           b_{(m-1)i}
         \end{bmatrix} = e_i = \begin{bmatrix}
           0 \\
           0 \\
           \vdots \\
           1 \\
           \vdots \\
           0
         \end{bmatrix} .
  \end{align*} 
Therefore, if we let $B=(b_{ij})_{0 \leq i, j \leq m-1}$ then $AB=I_{m}$. This shows that $A$ is invertible. 

\end{proof}
\bigskip

\begin{rem}
    When $r=1$, the matrix $A$ above is often called a circulant matrix (see \cite{Davis, CM1}). Such matrices have diverse applications in many scientific fields including engineering, physics, network theory, as well as pure and applied mathematics.
\end{rem}

Let us call the determinant $\Delta$ of (\ref{lem:det}) the unital determinant of the element
$a_0+a_1 \sqrt [m] r+a_2 (\sqrt [m] {r})^2+ \cdots + a_{m-1} (\sqrt [m] {r})^{m-1}$ in the
ring ${\Bbb Z}_p [ \sqrt [m] r]$. From Lemma \ref{lem:det}, we have the following

\begin{prop} Let $p$ be a prime number. If the unital determinant of any non-zero element of the
ring ${\Bbb Z}_p [\sqrt [m] r]$ is not zero, then ${\Bbb Z}_p [\sqrt [m] r]$ is a field. 
\end{prop}

\begin{cor}
    An element $a+b\sqrt{r}$ of the ring $\Z_n[\sqrt{r}]$ is a unit if and only if $a^2-rb^2 \in \Z_n^{\times}.$
\end{cor}
\begin{proof}This follows form the fact that the unital determinant of $a+b\sqrt{r}$ is $a^2-rb^2.$
\end{proof}
\begin{cor} \label{cor:field_det}
${\Bbb Z}_p[\sqrt r]$ is an integral domain (equivalently a field by  Lemma \ref{lem:finite_ring}) if and only if there does not exist $(a,b) \neq (0,0) \in \Z_p^{2}$ such that 
\[ a^2-rb^2=0.\]
In other words, $\Z_p[\sqrt{r}]$ is an integral domain if and only $c$ is not a square in $\Z_p.$
\end{cor}

\begin{rem}
    This corollary also follows from Theorem \ref{thm:fields_condition}. In fact, $\Z_p[\sqrt{r}]$ is not a field if and only if $x^2-r$ is reducible. In other words, it must have a linear factor; i.e. a root over $\Z_p$. This shows that $c$ must be a square. 
\end{rem}

The case $m=3$ is also quite interesting. In this case the unital determinant of the element $a_0+a_1 \sqrt[3]r + a_2 (\sqrt[3]r)^2$ is given by 

\begin{equation} \label{eq:cubic_unital_det}
\det \begin{bmatrix}
    a_0 & r a_2 &ra_1 \\ 
    a_1 & a_0 & ra_2 \\ 
    a_2 & a_1 & a_0 
\end{bmatrix} = a_0^3+ra_1^3+r^2a_2^3 -3ra_0a_1a_2.
\end{equation}

By Lemma \ref{lem:det}, we conclude that 
\begin{cor} \label{cor:unit_cubic}
    An element $a_0+a_1 \sqrt[3]{r}+a_2 (\sqrt[3]r)^2$ is a unit in in $\Z_n[\sqrt[3]{r}]$ if and only if 
    \[a_0^3+ra_1^3+r^2a_2^3 -3ra_0a_1a_2 \in \Z_n^{\times}. \] 
\end{cor}

\subsection{$m$-power map}
 We discuss a simple case where $x^m-r$ is reducible over ${\Z}_p$. 

\begin{prop}\label{quad55}
$r$ belongs to the image of the $m$-power map $f^{[m]}:{\Bbb Z}_p\to {\Bbb Z}_p$ if and only if $x^m-r$ has a linear factor. 
\end{prop}
\begin{proof}
Suppose, there is an element $a\in {\Bbb Z}_p$ such that
$a^m=r.$ This implies that $x-a$ a factor of $x^m-r.$ Conversely, if $x^m-r$ has a linear factor, say $x-a$, we must have $a^m=r.$
\end{proof}
\begin{cor}
    If the $m$-power map $f^{[m]}$ is onto then $x^m-r$ is reducible. 
\end{cor}

\begin{cor} \label{cor:m_less_than_3}
    Suppose that $m \leq 3.$ Then $x^m-r$ is reducible if and only if $r$ belongs to the image of the $m$-power map. 
\end{cor}
\begin{proof}
    If $r$ belongs to the image of the $m$-power map then by Proposition \ref{quad55}, $x^m-r$ is reducible. Conversely, suppose that $x^m-r$ is reducible. Since the degree of $x^m-r$ is at most $3$, at least one of its factors must be linear. By Proposition \ref{quad55}, we conclude that $r$ belongs to the image of the $m$-power map. 
\end{proof}

We have a quite interesting consequence of this corollary. 

\begin{prop}
    Suppose that $n=p$ is a prime number. Then $r$ belongs to the image of the $3$-power map if and only there exists $
    (a_0, a_1, a_2) \in \Z_n^3$ and $(a_0, a_1, a_2) \neq (0,0,0)$ such that $a_0^3+ra_1^3+r^2a_2^3 -3ra_0a_1a_2=0$. 
\end{prop}
\begin{proof}
    Suppose that $n$ belongs to the image of the $3$-power map, say $r=b^3$. We see that 
    \[a_0^3+ra_1^3+r^2a_2^3 -3ra_0a_1a_2= (a_0+ba_1+b^2a_2)(a_0^2+b^2a_1^2+b^4a_2^4-ba_0a_1-b^2a_0a_2-b^3a_1a_2). \]
    Therefore, if we take $a_0=-b, a_1=1, a_2=0$ then 
    \[ a_0^3+ra_1^3+r^2a_2^3 -3ra_0a_1a_2=0 .\] 
    Conversely, suppose that there exists a non-trivial triple $(a_0, a_1, a_2)$ such that $a_0^3+ra_1^3+r^2a_2^3 -3ra_0a_1a_2=0$. By Corollary \ref{cor:unit_cubic}, we know that $a_0 +a_1 \sqrt[3]r +a_2 (\sqrt[3]r)^2$ is not a unit in $\Z_n[\sqrt[3]r]$. In particular, this implies that $\Z_n[\sqrt[3]r]$ is not a field. Since $n$ is a prime number, by Theorem \ref{thm:fields_condition}, this also implies that $x^m-r$ is reducible over $\Z_n[x].$ By corollary \ref{cor:m_less_than_3}, $r$ must belongs to the image of the $3$-power map. 
\end{proof}

\begin{lemma} \label{lemma:m_power_map_onto}
    The $m$-power map is onto on $\Z_p$ if and only $\gcd(p-1, m)=1.$
\end{lemma}
\begin{proof}
    This follows from the fact that the unit group $\Z_p^{\times}$ of $\Z_p$ is cyclic of order $p-1.$
\end{proof}


\begin{prop}\label{quad555}
For any prime $p>2,$ there is an element $a\in {\Bbb Z}_p$ such that $x^2-a$ is irreducible over ${\Bbb Z}_p.$
\end{prop}
\begin{proof} 
By Lemma \ref{lemma:m_power_map_onto}, the $2$-power map is not onto.  Let $a$ be an element in $\Z_p$ which does not belong to the image of $f^{[2]}.$ Then by Corollary \ref{cor:m_less_than_3}, $x^2-a$ is an irreducible polynomial over ${\Bbb Z}_p.$ 
\end{proof}
\begin{rem}
    There are exactly $\frac{p-1}{2}$ elements $a \in \Z_p$ such that $x^2-a$ is irreducible. 
\end{rem}

\subsection{Gaussian Integers over ${\Bbb Z}_p$}
A prime $p$ is called Pythagorean if there are integers $a$ and $b$
such that $a^2+b^2=p$. It is well-known that this condition is equivalent to the condition that either $p=2$ or $p$ is a prime of the form $4k+1$ (see \cite{I}).

\begin{prop}
 ${\Bbb Z}_p[i]$ is a field if and only if $p$ is not an Pythagorean prime. 
\end{prop}
{\it Proof.}
If $p$ is Pythagorean then there are integers $a, b$  such that $a^2+b^2=p^2$. This implies that $(a+bi)((a-bi)=a^2+b^2\equiv p$ (mod $p$) i.e., $a^2+b^2\equiv 0$ mod $p$. This
implies that $a+bi$ has a zero divisor. Thus ${\Bbb Z}_p[i]$ is not an integral domain, i.e., it is not a field.

Conversely, suppose that $\Z_p[i]$ is not a field. By Corollary \ref{cor:field_det}, there exists $a,b \in \Z_p$ such that $a^2+b^2=0.$ This is equivalent to the condition that $p=2$ or $p \equiv 1 \pmod{4}.$ In other words, $p$ is a Pythagorean prime. 
\quad $\Box$

\begin{exa}
Since $(3+4i)(3-4i)=25\equiv 0$ mod $5,$ the ring ${\Bbb Z}_5[ {i}]$ is not a field. We can also see that  $x^2+1$ is reducible
over ${\Bbb Z}_5$. In fact $x^2+1 = (x+2)(x+3)$ in $\Z_5[x].$
\quad $\Box$\\
\end{exa}

\section{Finite rings over finite fields}

In this section, we study the finite ring $\F_q[\sqrt[m]{r}]$ where $\F_q$ is the finite field with $q$ elements. This section is a bit less elementary than the previous section in the sense that we need Galois theory for some of our arguments. In fact, we have the following statement about the Galois group of a finite field. 

\begin{prop} \label{prop:galois_finite_field}
    Let $\F_{q^n}/\F_q$ be a finite extension. Then the Galois group of $\F_{q^m}/\F_q$ is a cyclic group of order $n$ generated by the Frobenius automorphism $\text{Frob}_q$ defined by $\text{Frob}(x)=x^q.$ Note that $\text{Frob}_q$ is exactly the $q$-power map $f^{[q]}.$
\end{prop}
\begin{definition}
    Let $a,b$ be two relatively prime integers. The order of $a$ modulo $b$, denoted by $\ord_{b}(a)$ is the smallest number $h$ such that $a^h \equiv 1 \pmod{b}$. 
\end{definition}
We also need the following definition. 
\begin{definition}
    Let $\F_p$ be the prime finite field of characteristics $p$ and $r \in \overline{\F_p}^{\times}.$ The order of $r$, denoted by $\ord(r)$ is defined to be the smallest positive integer $h$ such that $r^h=1.$ 
\end{definition}
We are now ready to discuss the factorization of $x^m-r$ over a finite field $\F_q$ under the assumption that $\F_q$ contains a primitive $m$-roots of unity. We first observe the following phenomenon. 

\begin{thm} \label{thm:equal_degree}
    Suppose that $\F_q$ contains a primitive $m$-root of unity. Then $x^m-r$ can be factorized into a product of $\frac{m}{t}$ irreducible polynomials of degree $t$ where $t = \ord_{\ord(r) m}(q)$; namely, $t$ is the smallest number such that $ \ord(r) m |q^t-1.$ 
\end{thm}

\begin{proof}
    Since $\gcd(q,m)=1$, the polynomial $x^m-r$ is separable over $\F_q.$ Let $\alpha$ be any root of $x^m-r$ over $\overline{\F_q}.$ Then, by definition $\alpha^m=r.$ The finite extension $\F_q(\alpha)/\F_q$ is a Galois extension whose Galois group is generated by the Frobenius map $a \mapsto a^q$. Consequently, the Galois conjugates of $\alpha$ are $\{\alpha, \alpha^{q}, \ldots, \alpha^{q^{t-1}} \}$ where $t$ is the smallest positive integer such that $\alpha^{q^t}=\alpha.$ Equivalently, $\alpha^{q^t-1}=1.$ We remark that since the Galois orbit of $\alpha$ over $\F_q$ has $t$ elements,  the minimal polynomial of $\alpha$ over $\F_q$ must have degree $t.$ We claim that 
    $t= \ord_{\ord(r) m}(q)$, independent of $\alpha$. If this is true, our theorem is proved. In fact, since $\F_q$ contains a primitive $m$-root of unity, we know that $m|q-1.$ We then see that 
    \[ 1 = \alpha^{q^t-1} = (\alpha^m)^{\frac{q^t-1}{m}} = r^{\frac{q^t-1}{m}}.\]
    The above equality implies that $\ord(r)|\frac{q^t-1}{m}$ or equivalently $\ord(r) m |q^t-1.$ Since $t$ is the smallest positive integer with this property, we conclude that $t = \ord_{\ord(r) m}(q).$
\end{proof}

\begin{cor}
    Let $\F_q, r, m, t$ be as in Theorem \ref{thm:equal_degree}. Then $\F_q[\sqrt[m]{r}]$ is isomorphic to $\F_{q^t}^{\frac{m}{t}}.$
\end{cor}

We discuss a consequence of Theorem \ref{thm:equal_degree}. First, we recall the $k$-power map $f^{[k]}: \F_q \to \F_q$ defined by $
f^{[k]}(a)=a^k.$

\begin{thm} \label{thm: k_power}
    Suppose that $\F_q$ contains a primitive $m$-root of unity. Then $x^m-r$ is irreducible over $\F_q[x]$ if and only if $r$ does not belong to the image for each $k|m$ and $k>1.$
\end{thm}
\begin{proof}
    If $r$ belongs to the image of $f^{[k]}$ for some $k|m$ and $k>1$, then we can write $x=b^k$ for some $k>1$ and $k|m$ then we can see that $x^m-r$ is reducible since $x^{m/k}-b$ is a factor of $x^m-r.$ Conversely, suppose that $r$ does not belong to the image of $f^{[k]}$ for all $k|m$ and $k>1.$ We claim that $x^m-r$ is irreducible. In fact, let $t$ be the degree of an irreducible factor of $x^m-r$ over $\F_q[x]$ (by Theorem \ref{thm:equal_degree}, $t$ is independent of the choice of the irreducible factor). We claim that $t=m.$ We know that 
    \[ t = \ord_{\ord(r)m}(q) = \ord_{(q-1)m}(q).\] 
By definition, $t$ is the smallest positive integer such that 
$\ord(r) m | q^t-1.$ Put it equivalently 
\[ m|\frac{q^t-1}{\ord(r)} = \frac{q-1}{\ord(r)}(1+q+\ldots+q^{t-1}).\]
Since $\F_q$ contains a primitive $m$-root of unity, we know that $m|q-1.$ Consequently 
\begin{equation} \label{eq:congruence}
0 \equiv  \frac{q-1}{\ord(r)} (1+ q+\ldots+q^{t-1}) \equiv \frac{q-1}{\ord(r)} t \pmod{m}.
\end{equation}
Let $d=\gcd(m, \frac{q-1}{\ord(r)})=1.$ We claim that $d=1$. Indeed, since $\F_{q}^{\times}$ is a cyclotomic group of order $q-1$, there exists a primitive root in $g \in \F_q^{\times}$ such that 
$r = g^{\frac{q-1}{\ord(r)}}.$ Since $d|\frac{p-1}{\ord(r)}$, $r$ belongs to the image of $f^{[d]}.$ By our assumption on $r$, we must have $d=1.$ From this fact and Equation \ref{eq:congruence}, we conclude that $t \equiv 0 \pmod{m}$. Since $t|m$, we must have $t=m$. This shows that $x^m-r$ is irreducible. 
\end{proof}
We discuss two corollaries Theorem \ref{thm: k_power}. 

\begin{cor}
If $r$ is a primitive root in $\F_q^{\times}$ and $\F_q$ contains a primitve $m$-root of unity, then $x^m-r$ is irreducible over $\F_q[x]$.
\end{cor}
\begin{proof} Since $r$ is a primitive root and $m|q-1$, $r$ can not belong to the image of $f^{[k]}$ for all $k|m$ and $k>1.$ By Theorem \ref{thm: k_power}, $x^m-r$ must be irreducible. 
\end{proof}
The next corollary is a direct consequence of Theorem \ref{thm: k_power}.
\begin{cor}
Suppose that $m$ is a prime number and $\F_q$ contains a primitive $m$-root of unity. Then $x^m-r$ is irreducible if and only $r$ is not an $m$-power. 
\end{cor}

\begin{cor} \label{cor:counting}
    Suppose that $\F_q$ contains a primitive $m$-root of unity. Let $M$ be the  positive integer such that 
    \begin{enumerate}
        \item $m|M$ and $\text{rad}(M)=\text{rad}(m)$. Here, for a positive integer $a$, $\text{rad}(a)$ is the product of all distinct prime factors of $a.$
        \item $\gcd(M, \frac{q-1}{M})=1.$
    \end{enumerate}
    Then, the number of $r \in \F_q$ such that $x^m-r$ is irreducible is $\varphi(M) \dfrac{q-1}{M}.$
    
\end{cor}

\begin{proof} By Theorem \ref{thm: k_power}, we know that $x^m-r$ is irreducible if and only if $r$ does not belong to the image of $f^{[k]}$ for each $k|m$ and $k>1$. By the proof of this theorem, this is equivalent to saying that $\gcd(m, \frac{q-1}{\ord(r)})=1.$ Because $m|q-1$, this is also equivalent to $M|\ord(r).$ Let us write $\ord(r)=M h$ where $h|\frac{q-1}{M}.$ Since the group $\F_q^{\times}$ is cyclic, the number of elements of order $Mh$ is exactly $\varphi(Mh) = \varphi(M)\varphi(h).$ We conclude that the number of irreducible polynomials of the form $x^m-r$ is exactly 

\[ \sum_{h|\frac{q-1}{M}} \varphi(M) \varphi(h) = \varphi(M) \sum_{h|\frac{q-1}{M}}  \varphi(h) = \varphi(M) \frac{q-1}{M}. \]
The last equality follows from the fact that for any positive integer $n$ 
\[ \sum_{d|n} \varphi(d)=n.\]
\end{proof}
\begin{rem}
If we assume that $m$ is a squarefree number then we can relax the condition that $\F_q$ contains a primitive $m$-root of unity in Corollary \ref{cor:counting}. In fact, suppose that $\F_q$ does not contain a primitive $m$-root of unity. This implies that $m \nmid q-1.$ Because $m$ is a squarefree number, there exists a prime divisor $d$ of $m$ such that $d \nmid q-1.$ This implies that the $f^{[d]}$ power map is surjective. In other words, we can write $r=b^d$ for some $b \in \F_q.$ This implies that $x^{m/d}-b$ is a factor of $x^m-r.$ As a corollary, $x^m-r$ is always reducible. 
\end{rem}



\bigskip

\bigskip

\end{document}